\documentclass{article}
\usepackage[english]{babel}
\usepackage{amsfonts}
\usepackage{amsmath,amsthm,amssymb}
\usepackage{color}
\usepackage{authblk}
\usepackage{graphicx}
\usepackage{tikz}
\usepackage{stmaryrd}
\usepackage{mathrsfs,upgreek}
\usepackage{eucal}
\usepackage{enumerate}
\usepackage{changes}

\newtheorem{theorem}{Theorem}[section]
\newtheorem{lemma}[theorem]{Lemma}

\newtheorem{corollary}[theorem]{Corollary}
\theoremstyle{definition}

\newcommand{\op}[1]{\textrm{\upshape #1}}

\newcommand{\join}{\vee}

\newcommand{\meet}{\wedge}

\newcommand{\la}{\langle}
\newcommand{\ra}{\rangle}
\newcommand{\alg}[1]{{\textbf{\upshape #1}}}  %
\newcommand{\vv}[1]{\mathsf {#1}}

\newcommand{\f}{\varphi}

\renewcommand{\th}{\theta}

\newcommand{\D}{\Delta}

\newcommand{\sse}{\subseteq}

\newcommand{\HH}{{\mathbf H}}  
\newcommand{\SU}{{\mathbf S}} 
\newcommand{\PP}{{\mathbf P}}   
\newcommand{\VV}{{\mathbf V}}   

\newcommand{\ib}{\item[$\bullet$]}

\newcommand{\Con}[1]{\operatorname{Con}(\alg #1)}

\newcommand{\vuc}[2]{#1_1,\dots,#1_{#2}}

\newcommand{\imp}{\rightarrow}
\newcommand{\nneg}{\mathop{\sim}}

\makeatletter
\def\square{\RIfM@\bgroup\else$\bgroup\aftergroup$\fi
  \vcenter{\hrule\hbox{\vrule\@height.6em\kern.6em\vrule}\hrule}\egroup}
\makeatother
\newcommand{\Par}{\mathrel{\bindnasrepma}}

\newcommand{\smlcirc}{\raise3pt\hbox{\textrm{\circle{3.3}}}}
\newcommand{\smallcirc}{\mathrel{\smlcirc}}
\newcommand{\limp}{\relbar\joinrel\hspace{2.6pt}\smallcirc}

\newcommand{\one}{\textsf{\bf 1}}
\newcommand{\zero}{\textsf{\bf 0}}

\newcommand{\escl}{\text{\large$\mathop{!}$\normalsize}}
\newcommand{\?}{\text{\large$\mathop{?}$\normalsize}}

\newcommand{\myfrac}[2]{\dfrac{#1}{\lower.5ex\hbox{$#2$}}}
\mathchardef\hu="0362

\begin{document}

\title{An algebraic investigation of Linear Logic}
\author{Paolo Aglian\`{o}\\
DIISM\\
Universit\`a di Siena\\
Italy\\
agliano@live.com
}
\date{}
\maketitle

\begin{abstract} In this paper we investigate two logics from an algebraic
point of view. The two logics are: $\mathsf {MALL}$ (multiplicative-additive
Linear Logic) and $\mathsf {LL}$ (classical Linear Logic). Both logics turn out to
be strongly algebraizable in the sense of Blok and Pigozzi and their
equivalent algebraic semantics are, respectively, the variety of
Girard algebras and the variety of girales. We show that any variety of girales
has equationally definable principale congruences and we classify all
varieties of Girard algebras having this property. Also we investigate the
structure of the algebras in question, thus obtaining a representation theorem
for Girard algebras and girales. We also prove that congruence lattices of girales
are really congruence lattices of Heyting algebras and we construct examples
in order to show that the variety of girales contains infinitely many
nonisomorphic finite simple algebras.
\end{abstract}

This note is a reworking of a manuscript (which dates back to the late 1990's) that  had a very limited circulation at the time and was never published.
Lately it has been suggested that I made it  available on arxiv so I decided to review it, correcting the usual mistakes and modernizing it a little bit.
I wish to thank Carles Noguera and Wesley Fussner who encouraged me to complete this little project.

\section{Introduction}

When one wishes to investigate a nonclassical logic one has a
choice between two approaches: the syntactical  and the
algebraic.
The first usually gives rise to a {\em relational (Kripke-style) semantics},
while the other deals with {\em algebraic semantics}.
The great success of Kripke in the sixties with his  relational semantics for
modal and intuitionistic logic was a source of inspiration for many
researches based on his methods, while the algebraic approach
receded into the background. The algebraic approach became fashionable again
starting in the late seventies, mainly because of the work of W. Blok and
D. Pigozzi. W. Blok, in
his Ph.D. thesis \cite{Blok1976},  conducted an
in-depth study of Lewis' modal logic S4, and in \cite{Blok1980} he
investigated the entire {\em lattice of modal logics} by purely algebraic
means. Later he and D. Pigozzi investigated thoroughly the matter of
algebraizability of logics \cite{BlokPigozzi1989,BlokPigozzi1992}. This
investigation
set the foundation for a new field, now commonly called {\em abstract
algebraic logic}.

One of the first results of their line of investigation was the
identification  of the ``right'' concept of {\em algebraizable logic}.
Roughly speaking, a logic $\sf L$ is algebraizable if there is a class $\vv
K$
 of algebras (no
infinitary operations, no relations, no second
order axioms) which is to $\sf L$ what the variety of Boolean algebras is to
classical propositional calculus.  The class $\vv K$ is called the
{\em equivalent algebraic semantics}  of {\sf L}. The
knowledge that a given class $\vv K$ of algebras
is the equivalent algebraic semantics of a known logical system
yields a good deal of information on its algebraic structure.
Conversely, one can discover algebraic properties of members of $\vv K$
that can be transformed into logical data.

In this note we will apply this machinery to two logics:
{\em multiplicative-additive linear logic} ({\sf MALL}) and
{\em classical linear logic} ({\sf LL}).
 Both {\sf MALL} and {\sf LL} will turn out to be
algebraizable. As is  usually the case it is no surprise what
their equivalent classes are: they consist of residuated lattices (possibly with a modal operator) obeying equations reflecting the logical
axioms.

\section{Linear Logic}

Linear Logic
 is a ``resource-conscious'' logic
introduced by J.-Y. Girard in the late 80's \cite{Girard1987}. Since then
Linear Logic has been developed by Girard himself, by his school and by many,
many others  the  full list of whom would be too long to include here.

Linear Logic is resource conscious in that the left side of a sequent
represents a resource that cannot be used freely. In a Gentzen-style
axiomatization this consciousness shows itself by the absence of the
classical weakening and contraction rules. For instance $A,A \vdash B$
means that we use two resources of type $A$ to get a datum of type $B$.
Moreover---and this is the main difference from other substructural
logics---Girard introduced two operators (the {\em exponentials}) that serve
to allow
weakening and contraction in a controlled way on individual formulas.

The propositional language of Linear Logic consists of four families
of connectives\footnote{our notation is slightly different from the original formulation, in that $\zero$ and $\bot$ are exchanged; the reason is that the original formulation conflicts with the common usage in residuated lattices.}:
\begin{enumerate}
\item[$\bullet$] the {\em multiplicative} connectives: $\cdot$ ,
$\Par$ (the par, i.e. the parallel ``or''), $\imp$ (the linear implication,
Girard's $\limp$), $\zero$  and $\one$;
\item[$\bullet$] the {\em additive} connectives: $\join$, $\meet$, $\top$ and
$\bot$;
\item[$\bullet$] the {\em linear negation} $\neg$,
which is
a de Morgan
involution with respect to
$\join$ and $\meet$;
\item[$\bullet$] the {\em exponentials}: $\escl$ and $\?$.
\end{enumerate}
A suggestive way of thinking about how these connectives work is to view
formulas as data types. For instance $A \meet B$ is a datum from which we can
extract, once, either a datum of type $A$ and a datum of type $B$; $A\cdot
B$
is just a pair of data; $A\imp B$ is a method of transforming a single
datum of type $A$ into a datum of type $B$; $\escl A$  indicates that
we can extract as many data of type $A$ as we like (weakening and contraction
on the left side of a sequent); and so on.

The original formulation of Linear Logic is in a Gentzen-style
axiomatization. However, in order to take advantage of Blok-Pigozzi's theory
of algebraizability, it is helpful to look at Linear Logic as a 1-deductive
system in the sense of \cite{BlokPigozzi1992}.
When we turn a logical system into a deductive system, we use a
procedure that might not be in accordance with the original motivation for
introducing the logical system. This is patent in the case of substructural
logics, in that any deductive system satisfies all the structural rules
by definition.
In general, if a logical system is presented as the set of theorems
of a Hilbert-style formal system, then it  defines a 1-deductive system
in a standard way.
Here is the Hilbert-style axiomatization of the sistem $\mathsf{LL}$  as presented by A.
Avron \cite{Avron1988},
consisting  of twenty-four axioms and three inference rules.

\bigskip
\small
\begin{tabular}{ll}
$\text{(HL1)}\quad p \imp p$
&$\text{(HL2)}\quad (p \imp q)\imp((q \imp r)\imp (p \imp r))$\\[.1in]
$\text{(HL3)}\quad (p\imp(q\imp r)) \imp (q\imp(p\imp r))$
&$\text{(HL4)}\quad \neg\neg p \imp p$\\[.1in]
$\text{(HL5)}\quad (p \imp\neg q) \imp (q \imp \neg p)$
&$\text{(HL6)}\quad p \imp (q \imp p \cdot q)$\\[.1in]
$\text{(HL7)}\quad p\imp(q\imp r) \imp(p\cdot q \imp r)$
&$\text{(HL8)} \quad \one$\\[.1in]
$\text{(HL9)}\quad \one\imp(p\imp p)$
&$\text{(HL10)}\quad p\imp (\neg p\imp \zero)$\\[.1in]
$\text{(HL11)}\quad \neg \zero$
&$\text{(HL12)}\quad  p \meet q  \imp p$\\[.1in]
$\text{(HL13)}\quad  p \meet q \imp q$
&$\text{(HL14)}\quad (p\imp q) \meet (p\imp r) \imp (p \imp q \meet
r)$\\[.1in]
$\text{(HL15)}\quad  p \imp p \join q$
&$\text{(HL16)}\quad q \imp p \join q$\\[.1in]
$\text{(HL17)}\quad  (p\imp r) \meet (q\imp r) \imp (p\join q \imp r)$
&$\text{(HL18)}\quad p \imp \top$\\[.1in]
$\text{(HL19)}\quad \bot \imp p $
&$\text{(HL20)}\quad  q \imp (\escl p \imp q)$\\[.1in]
$\text{(HL21)}\quad  (\escl p \imp (\escl p\imp q))\imp (\escl p \imp q)$
&$\text{(HL22)}\quad  \escl(p\imp q) \imp (\escl p \imp \escl q)$\\[.1in]
\end{tabular}

\begin{tabular}{ll}
$\text{(HL23)}\quad \escl p \imp p$
&$\text{(HL24)}\quad  \escl p \imp \escl\escl p$ \\[.1in]
$\text{(MP)} \qquad \dfrac{p\quad p\imp q}{q}$
&$\text{(Adj)}\qquad \dfrac{p\quad q}{p \meet q}$\\[.1in]
$\text{(Nec)} \qquad \dfrac{p}{\escl p}$&
\end{tabular}

\bigskip
\normalsize

For the logical reasons why we need to introduce the (Adj) rule, we refer the reader to \cite{Avron1988}, p.171.
We call $\mathsf{MALL}$ the {\em multiplicative-additive linear logic}, i.e. the exponential-free fragment; it is of course
axiomatized by (HL1)-(HL19) plus (MP) and (Adj).
Moreover:
\begin{enumerate}
\item [-] In {\sf MALL} one can define $p \Par q$ as $\neg p \imp q$.
\item [-] In {\sf LL} one  can define $\? p$ as $\neg \escl(\neg p)$.
\item [-] Finally $\vdash_{\sf MALL} p\imp q$ if and only if $\vdash_{\mathsf{MALL}} \Longleftrightarrow
\neg(p\cdot\neg
q)$.
\end{enumerate}

\section{Consequence relations}

In his original paper \cite{Girard1987},
J.-Y. Girard gave absolutely no meaning to the concept
of ``linear logical theory'' or to
any kind of associated consequence relation.
The question was tackled  again  by A. Avron \cite{Avron1988}. He
observed that the classical
methods for associating a consequence relation to the Gentzen-type
presentation of Linear Logic gives rise to two meaningful consequence
relations.

\begin{enumerate}
\item [$\bullet$] (The {\em internal} consequence relation)
$$
\vuc {\f}n \vdash_{\sf LL}^I \psi
$$
iff the corresponding sequent is derived in the Gentzen-type formalism
iff $\f_1 \imp(\f_2 \imp(\dots(\f_n \imp \psi)\dots))$ is a theorem of
Linear Logic.
\item[$\bullet$] (The {\em external} consequence relation)
$$
\vuc {\f}n \vdash_{\sf LL}^E \psi
$$
iff the sequent $\Rightarrow \psi$ is derivable in the Gentzen-type
formalism
obtained from the linear one by adding $\Rightarrow \f_1,\dots\Rightarrow
\f_n$ as axioms.
\end{enumerate}
These two concepts are well known to coincide in classical and
intuitionistic logic but not for Linear Logic. However,
in Theorem 2.7 of \cite{Avron1988} A. Avron showed that\footnote{He really
proved it for the logic {\sf MALL} but the proof carries through once one
adds the exponentials.}
$\Sigma \vdash_{\sf LL} \f$ if and only if $\Sigma \vdash_{\sf LL}^E \f$, which
seems to imply that our view of Linear Logic as a deductive system is not
totally disconnected from its motivations.

\section{The algebrization}

While the exponentials $\escl$ and $\?$ were central in Girard's original
idea, the {\em exponential-free fragment} of Linear Logic has also attracted
a lot of interest.  Let us stress that {\sf MALL} is  a
honest-to-God substructural logic close to  the well-studied
system {\sf R} of {\em Relevance Logic} \cite{AndersonBelnap}.
This connection can be roughly expressed by the equation
$$
\text{ {\sf R} $-$ contraction =  {\sf MALL} $+$ distribution}.
$$
Techniques from {\sf R} have been applied to {\sf MALL} with some success
(see
for instance \cite{AllweinDunn1993}).
Moreover {\sf MALL} is superior to {\sf R} at least in that it has
a {\em cut-elimination} theorem. However, everything comes with a price
tag: the lack of distribution in {\sf MALL} makes things harder from
an algebraic point of view.

The relationship between {\sf MALL} and {\sf LL} is complicated.
The axioms (HL22)--(HL24) and (Nec) resemble  the introduction
of an S4 modality.
However (HL20) and (HL21) (weakening and contraction limited to exclamated
formulas) seem to be responsible for a quantum leap: {\sf MALL} is
decidable (really, {\em PSPACE-complete})
while {\sf LL} is not \cite{Lincoln1990}.

\begin{theorem}\label{MALL} Any fragment of  {\sf LL} contaning $\{\meet, \imp, \one\}$ is algebraizable with defining
equation
$p \meet \one = \one$ and congruence formulas $\D=\{p\imp q, q \imp p\}$.
\end{theorem}
\begin{proof}
The derivations
\begin{align*}
&\vdash p \mathrel{\D} p\\
&p\mathrel{\D} q  \vdash q\mathrel{\D}p\\
&p\mathrel{\D}q, q \mathrel{\D}r\vdash p\mathrel{\D} r
\end{align*}
follow readily from axioms (HL1)--(HL3) and
preservation by connectives is  easily checked.

It remains to show that $p\meet \one\mathrel{\D} \one \vdash p$ and that $p\vdash p\meet \one \mathrel{\D} \one$. Thus
for the first
\small
$$
\begin{array}{ll}
\one \imp p \meet \one&\text{(Hyp)}\\
\one &\text{(HL8)}\\
p\meet \one&\text{(MP)}\\
p &\text{(HL12)}
\end{array}
$$
\normalsize
and for the second
\small
\setlength{\arraycolsep}{.2 in}
$$
\begin{array}{ll}
\one \imp (p \imp p) &\text{(HL9)}\\
p\imp(\one \imp p) &\text{(HL3) + (MP)}\\
p &\text{(Hyp)}\\
\one \imp p &\text{(MP)}\\
\one \imp \one &\text{(HL1)}\\
(\one \imp p) \meet (\one \imp \one) \imp (\one \imp p\meet
\one)&\text{(HL14)}\\
\one \imp (p\meet \one) &\text{(MP)}\\
p\meet \one \imp \one &\text{(HL13)}\\
p\meet \one\mathrel{\D} \one
\end{array}
$$
\end{proof}

Let $\mathsf{T}$ be any fragment of $\mathsf{MALL}$ containing $\{\imp,\meet, \one\}$. The type of its EAS $\vv K_{\mathsf {T}}$ is determined by the connectives in $\sf T$ and we will follow the common usage of
denoting them by the same symbols. Moreover
$$
\Gamma \vdash_{\sf T} p \qquad \text{iff}
\qquad \{q \meet \one = \one : q \in \Gamma\} \vDash_{\mathsf{ K_T}} p \meet\one =
\one.
$$

A (pointed) {\bf Girard semilattice} is an algebra $\la A, \imp ,\meet, 1\ra$ where, $\la A, \meet, 1\ra$ is a pointed semilattice and moreover for all $a,b,c \in A$
\begin{align*}
&1 \imp a = a \tag{L1}\\
&a \imp a \ge 1 \tag{L2}\\
&(a \imp b) \meet (a \imp c) = a \imp (b \meet c) \tag{L3}\\
&a \imp b \le (c \imp a) \imp (c\imp b) \tag{L4}\\
&a \imp (b \imp c) \le  b \imp (a \imp c) \tag{L5}\\
& a \imp b, b \imp a \ge 1 \quad\text{implies}\quad a =b \tag{L6}
\end{align*}
Girard semilattices form a quasivariety $\mathsf{GS}$ that is not a variety (this can be shown by an easy reworking of known  examples).
Let us also observe that $\imp$ is  a BCI implication that is not a residuation (so a Girard semilattice is not in general a residuated semilattice in the sense of
\cite{NemWha1971}).

The following can be proved by standard techniques:

\begin{theorem} If $T=\{\imp, \meet, \one\}$ then $\mathsf{K_T}$ is the quasivariety  of Girard semilattices.
\end{theorem}

\begin{proof} Let $\alg A \in \mathsf{K_T}$. We define a relation on $A$ by setting
$$
a \le b \qquad  \text{iff}\qquad  (a \imp b) \meet 1 = 1.
$$
The relation is reflexive by (H1), transitive by (H2)  and
the congruence formulas imply antisymmetricity. Hence $\le $ is a partial order on $A$ and
(H12)--(H14) imply that
$a \meet b$ is the greatest lower bound  of $a$ and $b$, making $\alg A$  a pointed semilattice in which (L6) holds.

The rest consists of standard calculations; first observe that (L2) is a direct consequence of (HL9), (L3) comes in the same fashion from
(HL14), (L4) from (HL2) and (HL3) and (L5) from (HL3).   For (L1) we observe that the following derivations hold in $\sf T$:
\small

\setlength{\arraycolsep}{.2 in}
$$
\begin{array}{ll}
\one \imp (p \imp p) &\text{(HL9)}\\
p\imp(\one \imp p) &\text{(HL3) + (MP)}\\
\end{array}
$$
\small
\setlength{\arraycolsep}{.2 in}
$$
\begin{array}{ll}
(\one \imp p) \imp (\one \imp p) &\text{(HL1)}\\
\one \imp ((\one \imp p) \imp p )&\text{(HL3) + (MP)}\\
(\one \imp p) \imp p &\text{(HL8) + (MP))}\\
\end{array}
$$
\normalsize
Via the usual translation this implies (L1).
\end{proof}

Introducing the join causes no problems; an algebra $\la A,\imp,\join,\meet,1\ra$ is a {\bf Girard lattice} if
\begin{enumerate}
\ib $\la A, \join, \meet\ra$ is a lattice;
\ib $\la A,\imp, \meet,1\ra$ is a Girard semilattice;
\ib for all $a,b,c \in A$
\begin{equation}
(a \imp c) \meet (b \imp c) = (a \join b) \imp c. \tag{L6}
\end{equation}
\end{enumerate}
Clearly the equivalent algebraic semantics of the $\{\imp,\join\,\meet,\one\}$-fragment is the quasivariety $\mathsf{GL}$ of Girard lattices.

If we consider the equivalent algebraic semantics of the positive (i.e. without negation) fragment
of linear logic, then we have a BCI-implication that, thanks to (L6) and (L7), forms a residuated pair with $\cdot$. This implies that the equivalent algebraic semantics is is just the variety $\mathsf{CRL}$ of {\bf commutative residuated lattices}, with or without bounds.

\section{An embedding}\label{firstembedding}

In this section we would like to show that the positive fragment of $\sf MALL$ is not a conservative extension of the $\{\imp,\meet,\one\}$-fragment. The translation into algebraic terms consist in proving that
the $\{\imp,\meet,1\}$-subreducts of algebras in $\sf CRL$ form a proper subclass (as a matter of fact a subvariety) of $\sf{GS}$.
This is a consequence of the following well-known fact proved first in \cite{Agliano1996b} and rediscovered many times in the literature.

\begin{theorem} \cite{Agliano1996b} For every variety $\vv V$ of commutative residuated lattices, the class of $\{\meet,\imp,1\}$-subreducts of $\vv V$ is a variety.
\end{theorem}

It is very easy to find a quasiequation holding in the varieties of $\{\imp,\meet,1\}$-subreducts but not in $\sf GS$; in fact in $\sf CRL$, $\imp $ is a residuation and this somehow carries over
 in the sense that  for all $\{\imp,\meet,1\}$-subreduct $\alg A$ and for all $a,b \in A$
$$
a \imp b\ge 1 \quad\text{implies}\quad   a \le b.
$$
Now it is easily checked that this quasiequation does not hold in $\sf GS$ (since a Girard semilattice in not in general a residuated semilattice).

Let $\vv V$ be the variety of Girard semilattices satisfying the further equation
\begin{equation}
x \le ((x \imp y) \meet 1) \imp y\tag{L7}
\end{equation}
We claim that $\vv V$ is the variety of $\{\imp,\meet,1\}$-subreducts of $\sf CRL$.
Now it is easy to show that (L7) holds in any $\{\imp,\meet,1\}$-subreduct and implies (L6). Hence we only need to show that any member of $\vv V$ is embeddable in a commutative residuated lattice.
First let's prove that the algebras in $\vv V$ have {\em residuals without residuations}:

\begin{lemma}\label{reslemma} Let $\alg A \in \vv V$ and let $a,b,c \in A$; then
\begin{enumerate}
\item $a \le b$ if and only if $a \imp b \ge 1$;
\item $a \le (a \imp b) \imp b$;
\item $a \le b$ implies  $b \imp c \le  a \imp c$ and $c \imp a \le c \imp b$.
\end{enumerate}
\end{lemma}
\begin{proof} Suppose $a \le b$; then $a \meet b = a$. Then by (L3)
$$
(a \imp a) \meet (a \imp b) = a \imp (a \meet b) = a \imp a;
$$
so $a \imp b \ge a \imp a \ge 1$ by (L2).

Conversely, assume $a \imp b \ge 1$; then by (L7) and (L1)
$$
a \le ((a \imp b) \meet 1) \imp b = 1 \imp b = b.
$$

For (2), by (L5) and (L2) we get
$$
a \imp ((a \imp b) \imp b) = (a \imp b) \imp (a \imp b) \ge 1
$$
and by (1) $a \le (a \imp b ) \imp b$.

The proof of (3) is routine using (1), (L3) and (L5).
\end{proof}

The embedding we are going use is based on the theory of frames developed in \cite{Dunn1993}.  Let $\alg A \in \vv V$ and let $\Gamma_\alg A$ be the set of semilattice filters of $\alg A$; we say that a subset $X \sse \Gamma_\alg A$ is {\bf hereditary} if for all $F,G \in \Gamma_\alg A$,  $F \in X$ and $F \sse G$ implies $G \in X$. We also define for $a \in A$,  $\mathbf a = \{F \in \Gamma_\alg A: a \in F\}$ and we note that $\mathbf a$ is hereditary. Note that the the intersection of any family of hereditary subsets is hereditary; so we can define a closure operator in which the closed subsets are precisely the hereditary subsets of $\Gamma_\alg A$. It follows that the hereditary subsets of $\Gamma_\alg A$ form an algebraic lattice $D(\alg A)$ ordered by inclusion.

 Next we define a ternary relation on $\Gamma_\alg A$; for $F,G,H \in \Gamma_\alg A$
$$
R(F,G,H)\qquad\text{if and only if}\qquad \text{for all $a,b \in A$, $a \in F$ and $a \imp b \in G$ implies $b \in H$}.
$$
This relation allows us to introduce additional operations: if $X,Y \in D(\alg A)$
\begin{align*}
&X \circ Y = \{H: \exists F \in Y, \exists G \in X \ \text{with}\ R(F,G,H) \}\\
&X \imp Y = \{H: \forall F,G \ \text{if $R(F,H,G)$ and $F \in X$, then $G \in Y$}\}.\\
\end{align*}
Of course the relationships between equations satisfied in $D(\alg A)$ and the properties of $R$ are relevant. In \cite{Dunn1986} there is a long list of these correspondences  with no proofs, simply quoting the work of R. Routley and R. Meyer on the semantics of entailment \cite{RoutleyMeyer1973}; some proofs are indeed there, but they are so embedded in the general abstract theory of entailment that their connection to this algebraic setting is not immediately  clear. That's why here we prefer to present direct proofs.

\begin{lemma} For any $\alg A \in \vv V$, $\la D(\alg A),\circ,\one\ra$ is a commutative monoid.
\end{lemma}
\begin{proof} Proving that $D(\alg A)$ is closed under $\circ$ is straightforward.
 Let then $X,Y \in D(\alg A)$ and suppose that $H \in X \circ Y$; then there is an $F \in Y$ and a $G \in F$ with $R(F,G,H)$. Let $a \in G$ and $a \imp b \in F$; then (by Lemma \ref{reslemma} $a \le (a \imp b) \imp b \in G$. So since $a \imp b \in F$ and $R(F,G,H)$ we get $b \in H$; so $R(G,F,H)$ holds and hence $H \in Y \circ X$. This shows that $\circ$ is commutative.

Let $\nabla_A$ be the {\em positive cone} of $\alg A$, i.e. the principal filter generated by $1$. Note that
for any $F \in \Gamma_\alg A$, $R(F,\nabla_A,F)$ holds and, since
$\nabla_A \in \one$, we get at once that $ X \subseteq \one \circ  X$. Conversely, let
$H \in \one \circ X$; then there is an $F \in X$ and a $G \in \one$ with
$R(F,G,H)$. If $a \in F$, then $a \imp a \ge 1 \in G$ and hence
$a \in H$, so that $F\sse H$. But $X$ is hereditary and $F \in X$, so
$H \in X$ and eventually $\one\circ X = X$.

Associativity requires more work.  Let $X,Y,Z \in D(\alg A)$ with with $H \in (X\circ Y)\circ Z$; then there is an $F \in Z$ and a $U \in X\circ Y$ with $R(F,U,H)$ and  a $K \in X$ and a $G \in Y$ with $R(G,K,U)$.
Let $L=\{ d \in A: b \le a \imp d\ \text{for some}\ a \in F, b \in G\}$; then using Lemma \ref{reslemma} we can show that $L \in \Gamma_\alg A$ and  clearly $R(F,G,L)$.   Assume now $d \in L$ and $d \imp c \in K$; again by
Lemma \ref{reslemma}, (L4) and (L5) we get
$$
d \imp c \le (a \imp d) \imp (a \imp c)
$$
 so $(a \imp d) \imp (a \imp c) \in K$. Since $d \in L$ there are $a \in F$ and $b \in G$ with $b \le a \imp d$, so $a \imp d  \in G$ and, since $R(G,K,U)$, we get $a \imp c \in U$. But $a \in F$, $a \imp c \in U$ and $R(F,U,H)$ implies $c \in H$. Hence we conclude that $R(L,K,H)$.

Now by definition $L \in Y \circ Z$ and hence, since $R(F,G,L)$,  $H \in X \circ (Y \circ Z)$; we have thus proved that $(X \circ Y) \circ Z \sse X \circ (Y \circ Z)$. The opposite inclusion follows from a similar argument, hence
$\circ$ is associative.
\end{proof}

\begin{lemma} For each $\alg A \in \vv V$, $(\imp,\circ)$ form a residuated pair w.r.t. the lattice ordering of $D(\alg A)$.
\end{lemma}
\begin{proof} Since we already know that $D(\alg A)$ is closed under $\circ$ we have only to check that it is closed under $\imp$ as well. Let then $H \in X \imp Y$; then if $H \sse H'$ and $a \imp b \in H'$, then
$a \imp b \in H$. It follows that if $R(F,H',G)$, then $R(F,H,G)$ for all $F,G \in \Gamma_\alg A$ and so $H' \in X\imp Y$ which is then hereditary.

Next we have to show that
$$
X \circ Y \sse Z \qquad\text{if and only if}\qquad X \sse Y \imp Z.
$$
Assume then that $X \circ Y \subseteq Z$ and let $H \in X$. Let $F \in Y$ such that $R(F,H,G)$; then by definition $G \in X\circ Y$ and hence $G \in Z$. But this implies $H \in Y \imp Z$, as wished.  Conversely suppose $X \le Y \imp Z$ and let $H \in X \circ Y$; then there are $F \in Y$ and $G \in X$ with $R(F,G,H)$. But then $G \subseteq Y \imp Z$, so if $F \in Y$ and $R(F,G,H)$, then $H \in  Z$ as wished.
\end{proof}

Hence we have shown that:

\begin{theorem} For any $\alg A \in \vv V$, $\alg D(\alg A) =\la D(\alg A), \imp,\join, \meet,\circ, \one\ra$ is commutative residuated lattice.
\end{theorem}

Finally we prove the embedding.

\begin{theorem} Any algebra $\alg A \in \vv V$ is embeddable in $\alg D(\alg A)$.
\end{theorem}

\begin{proof}
Define a mapping $h:\alg A \longmapsto H(\alg A)$ by
$$
h(a) = \mathbf a.
$$
We start showing that for any $a,b \in A$, $H \in h(a\imp b)$ if and only if $H \in \mathbf a \imp \mathbf b$. This is equivalent to showing that, for $a,b \in A$
$$
a \imp b \in H\qquad\text{if and only if}\qquad\text{$\forall F,G$ if $R(F,H,G)$ and $a \in F$, then $b \in G$}.
$$
The left-to-right implication is a straightforward consequence of the definitions. Assume now that $a \imp b  \notin H$; we will show that there exists $F,G$ with $a \in F$, $R(F,H,G)$ but $b \notin G$.
Let's denote by $[a)$ the principal filter generated by $a$. Let $F=[a)$ and $G= \{d: \ \text{there is a}\ c \in H\ c\le a \imp d\}$; note that $b \notin G$ otherwise $c \le a\imp b$ for some $c \in H$ and since $H$ is a filter we would have $a\imp b \in H$, contrary to the hypothesis. Again we can show that $G$ is a filter using Lemma \ref{reslemma}; so $G \in \Gamma_\alg A$, $a \in F$ and $b \notin G$. Now we show that $R(F,H,G)$;  if $u \in F$ and
$u \imp v  \in H$, then $a \le u$ and hence $u \imp v\le a \imp v$,  which by definition implies $v \in G$ and thus  $R(F,H,G)$.

Now $a\le b$ implies  that $\mathbf a \sse \mathbf b$, so $h$ is order preserving; hence to conclude the proof it is enough to show that  $h$ is injective. But if $\mathbf a = \mathbf b$, since
$[a) \in \mathbf a$, we get $[a) \in \mathbf b$ so $b \in [a)$; by the same fashion $a \in [b)$ and hence $a =b$.
\end{proof}

\section{Algebraizing $\sf MALL$ and $\sf LL$}

For the complete $\sf MALL$ we have to work a little bit more.  An element $a$ of a commutative residuated lattice is {\bf involutive} if for all $b \in A$, $(b \imp a) \imp a = b$; it is well known
 \cite{BusanicheCignoli2009}  that if in case we define  $\nneg b := (b \imp a) \imp a$, then $\nneg$ is a  negation with the following properties
\begin{enumerate}
\ib $\nneg\nneg a = a$   (involutive);
\ib  $\nneg (a \join b) = \nneg a \meet \nneg b$ and $\nneg(a\meet b) = \nneg a \join \nneg b$  (De Morgan);
\ib  $a \le b$ implies $\nneg b \le \nneg a$  (antitonic);
\ib  $\nneg(a\cdot \nneg b) = a \imp b$  (contraposition).
\end{enumerate}
It is easily seen that these properties are not independent; for instance any negation that is involutive and satisfies one of the other three properties must satisfy them all.

A structure $\la A,\join,\meet, \imp,\cdot,0,1\ra$ is a {\bf Girard algebra} if
\begin{enumerate}
\ib $\la A, \join,\meet, \imp,\cdot,1\ra$ is a commutative residuated lattice;
\ib $0$ is an involutive element.
\end{enumerate}
In this case $\nneg x = x \imp 0$ is an involutive and antitonic so it also De Morgan and satisfie contraposition; so $\imp$ and $\cdot$ are definable in terms of each other.
Conversely any commutative residuated lattice with a negation $\nneg$ that is involutive and De Morgan can be seen as a Girard algebra upon defining
$0 = \nneg 1$. A {\bf bounded Girard algebra} is a Girard algebra with an additional constant $\top$ satisfying $x \imp \top \ge 1$; we define $\bot := \nneg \top$.  By the usual standard arguments we get:

\begin{theorem} The equivalent algebraic semantics of $\sf MALL$ is the variety of bounded Girard algebras.
\end{theorem}

Note that there $0$ and $1$ can be in any ordering relation; in particular it may happen that $1 \le 0$, which implies that $\sf MALL$ is not an explosive logic or, alternatively, it is a paraconsistent logic. It is worth noting that explosivity in our case means that
$$
0 = 0 \cdot 1 = 0 \cdot \nneg 0 \le a
$$
for all $a$. This implies $0=\bot$ and hence $1= \top$; hence the minimal explosive extension of $\sf MALL$ has as equivalent algebraic semantics the variety of integral Girard algebras. Since in this case the negation is a orthocomplementation  it can be also seen as the variety of residuated ortholattices.

If we look at $\sf LL$ it is clear that adjoining the exponentials corresponds to considering Girard algebras superimposed with a certain S4 modality. A {\bf girale} is an algebra $\la A,\join,\meet,\imp,\cdot,\nneg,1, \escl\ra$
where
\begin{enumerate}
\ib $\la A,\join,\meet,\imp,\cdot,\nneg,1\ra$ is a Girard algebra;
\ib $\escl$ is unary and for all $a,b \in A$
\begin{align*}
&\escl 1 = 1 \tag{G1}\\
&\escl a \le a \meet 1 \tag{G2}\\
&\escl a  \escl b = \escl(a \meet b) \tag{G3}\\
&\escl\escl a = \escl a. \tag{G4}
\end{align*}
\end{enumerate}

Let's prove some algebraic properties of girales.

\begin{lemma}\label{girlemma} Let $\alg A$ be a girale; then for any $a,b,c \in A$
\begin{enumerate}
\item $a \le b$ implies $\escl a \le \escl b$;
\item $b \le \escl a \imp b$;
\item $\escl a = \escl a \escl a$;
\item $ab \le c$ implies $\escl a \escl b \le \escl c$;
\item $a \ge 1$ implies $\escl a =1$;
\item $\escl(\escl a \escl b) = \escl a \escl b \le \escl(ab)$;
\item $\escl(a \imp b) \le \escl a \imp \escl b$;
\item $\escl a \imp (\escl a \imp b) \le \escl a \imp b$.
\end{enumerate}
\end{lemma}
\begin{proof} (1) is immediate from (G2) and (G3); next note that in any commutative residuated lattice $b \le (a \meet 1) \imp b$, thus (2) follows from (G2), while (3) is again a straightforward consequence
of (G3).

If $ab \le c$, then $\escl (a \meet b) = \escl a \escl b \le c$; hence (4) follows from (1) and (G4) while (5) follows from (1), (G1) and (G2).  For (6) we compute
$$
\escl (\escl a \escl b) = \escl \escl (a \meet b) = \escl(a \meet b) = \escl a \escl b.
$$
Moreover $\escl a\escl b \le ab$ so
by (1) $\escl(\escl a \escl b) \le \escl(ab)$ and hence $\escl a \escl b \le \escl(ab)$.
Next since $\escl a \le a$, we get $a \imp b \le \escl a \imp b$;
so $(a\imp b) \escl a \le b$, so $\escl (a \imp b) \escl a \le b$ and by (6) $\escl (a \imp b) \escl a \le \escl b$.  So (7) holds.

For (8) we observe that in any  residuated lattice, if $a$ is an idempotent element in a residuated lattice then for all $b$
$$
a(a \imp (a \imp b)) =a^2(a \imp (a \imp b)) \le b
$$
and by residuation
$$
a \imp (a \imp b) \le a \imp b.
$$
Since $\escl a$ is idempotent by (3) we conclude that
$$
\escl a \imp (\escl a \imp b) \le  (\escl a \imp b)
$$
so (8) follows.
\end{proof}

Now using Lemma \ref{girlemma} and the usual techniques of algebrization of logical systems it is straightforward to show that:

\begin{theorem} $\sf LL$ is strongly algebraizable and its equivalent algebraic semantics is the variety $\vv G$ of bounded girales.
\end{theorem}

So in a way girales have the same relationship to $\sf MALL$  that interior algebras have to classical logic.
More generally they belong to the very general class of (residuated) lattices with a superimposed modality.

\section{Another embedding theorem}

We will show that $\sf LL$ is a conservative extension of $\sf MALL$. Of course we will do it from the algebraic side, i.e. we will prove that the class of subreducts of girales to the type of Girard algebras is the variety of Girard algebras.  This is equivalent to showing that any Girard algebra is embeddable in a girale; in order to do so we will collect several information, that will be useful for other investigations as well, on the algebraic structures of girales.

Let $\alg P$ be any poset; we say that $Q$ is a {\bf relatively complete} subset of $P$ if for all $p \in P$
$$
\sup\{q \in Q: q \le p\}\qquad \inf\{q \in Q: q \le p\}
$$
both exist.

Let $\alg A$ be a Girard algebra and let as usual $A^- = \{ a : a \le 1\}$;  a {\bf relatively complete Heyting} subset of $A$ is a subset $H \sse A^-$ with the following properties:
\begin{enumerate}
\item  $1 \in H$;
\item $H$ is a relatively complete subset of $A$;
\item $H$ is closed under multiplication;
\item for all $a \in H$, $a^2 = a$.
\end{enumerate}

\begin{lemma}\label{relheyting} Let $\alg A$ be a Girard algebra and let $H$ be a relatively complete Heyting subset of $A$. If we define $\escl_{\small H} a = \sup\{b \in H: b \le a\}$, then
$\la \alg A, \escl_{\small H}\ra$ is a girale. Conversely if $\alg A$ is a girale then  $H = \escl A = \{\escl a: a \in A\}$ is a relatively complete Heyting subset of
$\alg A$ and $\escl_{\small H} a = \escl a$.
\end{lemma}
\begin{proof} We have to check that $\escl_H$ satisfies (G1)-(G4);  (G1) is obvious since $H \subseteq A^-$ and $1 \in H$ and (G2) follows from the definition of $\escl_H$.
Since clearly $a \le b$ implies $\escl_H a \le \escl_H b$ and $\escl_H a \escl_H a = \escl_H a$, from $\escl_H(a \meet b) \le \escl_H a, \escl_H b$ we get $\escl_H(a\meet b) \le \escl_H a \escl_H b$.
For the converse, note that $\escl_H a \le a \meet 1$ and $\escl_H b \le b \meet 1$; since $H$ is closed under multiplication we get
$$
\escl_H a \escl_H b = (\escl_H a \escl_H b)(\escl_H a \escl_H b) \le (a \meet 1)(b \meet 1) \le a \meet b \meet 1.
$$
This proves that $\escl_H a \escl_H b \le \escl_H (a \meet b)$ and hence (G3). Finally (G4) is obvious from the definition of $\escl_H$.

Since $1 \in \escl A$ by (G1), $\escl A\sse A^-$ by (G2),  it is closed under products by (G3) and consists of idempotents by Lemma \ref{girlemma}(3),  we need only to show that
$$
\escl a =\sup\{\escl b: \escl b \le a\}.
$$
Now if $b \le \escl a$, then $\escl b \le \escl a$ by Lemma \ref{girlemma}(1); so $\escl a$ is an upper bound.  Let $\escl b \le c$ for all $\escl b \in \escl A$ such that $\escl b \le a$; since $\escl a \le a$ we get that
$\escl a \le c$, so $\escl a$ is the least upper bound.
\end{proof}

\begin{corollary} Every complete Girard algebra is embeddable in (as a matter of fact, it is a reduct of) a girale.
\end{corollary}
\begin{proof} In this case the set $H= \{a: a\le 1\ \text{and}\  a^2=a\}$ is a nonempty relatively complete Heyting subset of $A$ and Lemma \ref{relheyting} applies.
\end{proof}

So to prove that Girard algebras are exactly subreducts of girales it is enough to show that any Girard algebra can be embedded in a complete Girard algebra. Now we could be tempted to use the same embedding
we used for Girard lattices; as a matter of fact it is not hard to show that if $\alg A$ is a commutative residuated lattice then $\alg A$ is embeddable in the complete and commutative residuated lattice $D(\alg A)$
(see \cite{Agliano1996b} for details). However introducing an involutive De Morgan negation causes problems; these problems can of course be solved by constructing a different embedding using for instance the circle
of ideas in \cite{AllweinDunn1993}, but we have a different and more direct embedding that does the job and we proceed to illustrate it.

Let $\alg A$ be a Girard algebra; we define a binary relation $R$ on $A$ by $(a,b) \in R$ if and only if $\nneg b \not\ge a$; this relation is symmetric since $\nneg$ is De Morgan. As for all binary relations
there is a closure operator $Q$ naturally associated to it; if $U \sse A$ we can define $Q(U) = \{a: (u,a) \in R\ \text{for some}\ u \in U\}$. It is a standard exercise to prove that $Q$ is a closure
operator on $A$ and hence the closed sets form a complete lattice with universe $C(\alg A)$.

\begin{lemma} Let $\alg A$ be a girard algebra; then $\alg C(\alg A) = \la C(\alg A),\join,\meet, \imp, \cdot, \nneg, \mathbf 1\ra$ is a complete Girard algebra upon defining for $X,Y \in C(\alg A)$
\begin{align*}
&X \cdot Y = Q(\{ab: a \in X, b \in Y\}) \\
&\nneg X = \{b: \nneg c \not\ge b\ \text{implies}\ c \ge a\ \text{for all}\  a \in X\}\\
&X \imp Y = \nneg(X \cdot \nneg Y)\\
&\one = Q(1).
\end{align*}
\end{lemma}

By our previous discussion to prove the lemma it is enough to show that the negation defined above is involutive and satisfies contraposition; this is a simple exercise, using the analogous properties of the negation in Girard algebras, and we leave it to the reader.
The next lemma is more important.

\begin{lemma} Let $\alg A$ be a Girard algebra and for $a \in A$ let's denote by $(a]$ the principal ideal generated by $a$. Then
\begin{enumerate}
\item for any $a \in A$, $(a]= Q(a) \in C(\alg A)$;
\item the mapping $a \longmapsto (a]$ is an embedding of $\alg A$ in $\alg C(\alg A)$.
\end{enumerate}
\end{lemma}
\begin{proof} Observe that
$$
 Q(a) = \{b: c \not\ge b  \ \text{implies}\ c\not\ge a\}.
$$
Suppose that $b \notin (a]$; then  $b \not \le a$ and then $b \notin Q(a)$. Conversely if $b \notin Q(a)$,
then there exists a $c$ with $\nneg c \not\ge b$,
and $\nneg c \ge a$; hence $b \not\le a$ and  so
$b \notin (a]$. This proves (1).

For (2) it is obvious that the mapping is a meet homomorphism.
Let's  show that
$$
(a] \join (b] = (a \join b]
$$
 Observe that $(a] \join (b] = Q((a] \cup (b])$. Since
$(a] \cup (b] \sse (a \join b]$ and the latter is closed, one inclusion is
clear. Next observe that
$$
(a]\join (b] = \{c: \nneg d \not\ge c \ \text{implies}\ (\nneg d\not\ge e\ \text{for
some}\ e \le a\ \ \text{or}\ \ \nneg d\not\ge f\ \text{for some}\ f \le b)\}.
$$
Let  $c \notin (a] \join (b]$. Then there exists a $d$
with $\nneg d \not\ge c$ but $\nneg d \ge a$ and $\nneg d \ge b$. Hence
$\nneg d \ge a \join b$ and so, since $\nneg d \not\ge c$,  $c \notin (a \join b]$.

A similar argument shows that
$$
(a]\cdot (b] = Q(\{uv:  u\le a, \ v \le b\}) = (ab].
$$
Next we check that
$$
(\nneg a] =\nneg(a] = \{ b: c \not \ge b \ \text{implies}\ c \not\le a\}.
$$
Suppose that $b \notin \nneg(a]$. Then there
exists a $c$ with $\nneg c \not\ge b$ and $c \le a$. Thus $b \not\le \nneg a$,
otherwise
$b \le \nneg a \le \nneg c$.
Conversely if  $b\not\le \nneg a$, then $\nneg a \not\ge b$ and thus
 $b \notin \nneg(a]$.

Since $\imp$ is definable in both cases by the negation the mapping is a homomorphism and it is obviously injective.
\end{proof}

\begin{corollary} Every (bounded) Girard algebra is embeddable in a (bounded) girale; hence the variety of (bounded) Girard algebras is exactly the class of subreducts
of (bounded) girales. Therefore $\sf LL$ is a conservative extension of $\sf MALL$.
\end{corollary}

\section{Congruences}

Congruences in commutative residuated lattices are well known; since the variety $\sf CRL$ of commutative residuated lattices is ideal-determined in the sense of \cite{AglianoUrsini1992} the congruences ar completely determined by certain subsets of the universe that we call {\bf U-ideals} (Ursini ideals). In case of $\sf CRL$ these subsets have a particularly transparent description; if $\alg A \in \vv{CRL}$ a {\bf filter} of $\alg A$ is a subset $F \sse A$ such that
\begin{enumerate}
\ib $F$ is a lattice filter;
\ib $1 \in F$;
\ib if $a, a \imp b \in F$ , then $b \in F$.
\end{enumerate}
Since the intersection of any family of filters is clearly a filter, there is a closure operator on $A$ in which the closed sets are exactly the filters; the operator is easily shown to be algebraic,  so the filters of $\alg A$ form an algebraic lattice $\op{Fil}(\alg A)$. The following fact was observed first in \cite{Agliano1996b} and has been rediscovered many times since.

\begin{theorem}\label{confil} If $\alg A \in \vv{CRL}$ then $\Con A$ and $\op{Fil}(\alg A)$ are isomorphic through the mappings $\th \longmapsto 1/\th$ and
$F \longmapsto \th_F = \{(a,b): a \imp b, b \imp a \in F\}$.
\end{theorem}

Now it is evident that a congruence of a Girard algebra is a congruence of its underlying commutative residuated lattice structure, so Theorem \ref{confil} specializes easily to Girard algebras.  What about girales? They are still  ideal-determined, so the congruences are totally determined by the U-ideals. However the operation
$\escl$ is a {\em compatible operation} in the sense of \cite{Agliano1996b}; this in turn implies that the U-ideals of a girale $\alg A$ are just the filters of its underlying Girard algebra, that are closed under $\escl$.
We will name these subsets  {\em filters} as well and let the context clear the meaning. In any case Theorem \ref{confil} holds; if $\alg A$ is a girale its congruence lattice is isomorphic with the lattice of filters.
Moreover we get the following useful description:

\begin{lemma}\label{filgen} Let $\alg A$ be a girale, let $X \sse A$ and let $\op{Fil}_\alg A(X)$ the filter generated by $X$; then
$$
\op{Fil}_\alg A(X) = \{a: \escl b_1 \dots \escl b_n \le a\ \text{for some}\ \vuc bn \in X\}.
$$
\end{lemma}
\begin{proof}  Let $F$ be the set described by the right hand side of the equality. Then  $F$ is a lattice filter, since it is the union of a directed family of lattice filters; moreover if
$a \in F$, then there are $\vuc bn \in X$ such that
$$
\escl (b_1 \meet \dots \meet b_n) = \escl b_1\dots \escl b_n \le a.
$$
This implies immediately that $\escl a \in F$, so $F$ is a filter which contains $X$. If $G$ is another filter containing $X$ then $F \sse G$, so $= \op{Fil}_\alg A(X)$.
\end{proof}

Lemma \ref{filgen} has obvious consequences; let $[a)$ denote the principal filter generated by $a$.

\begin{corollary}\label{cor1} Let $\alg A$ be a girale; then for any $a \in A$, $\op{Fil}_\alg A(a) = [\escl a)$. In other words  every principal filter of $\alg A$ is principal as a lattice filter.
\end{corollary}

\begin{corollary} The variety $\sf Gi$ of girales has equationally definable principal congruences.
\end{corollary}
\begin{proof} For any girale $\alg A$ and for $a,b \in A$ let $a \leftrightarrow b:= (a \imp b) \meet (b \imp a)$ and by $[a)$ the principal lattice filter generated by $a$. Then by $\op{Fil}_\alg A(a \leftrightarrow b) = [\escl(a \leftrightarrow b)$; hence $ c \leftrightarrow d \in \op{Fil}_\alg A(a \leftrightarrow b)$ if and only if $\escl(a \leftrightarrow b) \le c \leftrightarrow b$. But using the isomorphism between the congurence lattice and the filter lattice it is easily seen that $\op{Cg}_\alg A(a,b) = \op{Cg}_\alg A(1,\escl(a \leftrightarrow b))$. So we get that
$$
(c,d) \in \op{Cg}_\alg A(a,b)\qquad\text{if and only if}\qquad \escl(a \leftrightarrow b) \le c \leftrightarrow d
$$
i.e.  $\vv{Gi}$ has equationally definable principal congruences.
\end{proof}

Through the general theory of algebraizable logics we get that $\sf LL$ has the deduction theorem: $\Gamma, p \vdash_{\sf LL} q$ if and only if $\Gamma \vdash_{\sf LL} \escl p \imp q$. But this is of course
well known \cite{Avron1988}.

From Lemma \ref{relheyting} we know that $\escl A$ is a relatively complete Heyting subset of $\alg A$; with some adjustment it can be made into a Heyting algebra
$\escl \alg A$ which is deeply connected with $\alg A$.

\begin{theorem} For any girale $\alg A$, $\escl \alg A = \la \escl A, \join, \meet_{!},\imp_{!},\bot,1\ra$ is a Heyting algebra, where for $u,v \in \escl A$
$$
u \meet_{!} v = u \meet v\qquad  u \imp_{!}v = \escl (u \imp v).
$$
Moreover $\Con A \cong \op{Con}(\escl \alg A)$.
\end{theorem}
\begin{proof} We observe first that $\escl$ is a {\em conucleus} in the sense of \cite{MontagnaTsinakis2010}; then we apply Lemma 3.1 in \cite{MontagnaTsinakis2010} to conclude that
 $\escl \alg A$ is a commutative residuated lattice. Since $\escl 1  = 1$, it is also integral and by Lemma \ref{girlemma}(3) every element
of $\escl \alg A$ is idempotent. This is enough to deduce that $\escl \alg A$ is a Heyting algebra.

Next  we  show that the  the mappings
$$
H \longmapsto H \cap \escl A \qquad\qquad
G \longmapsto \op{Fil}_{\alg A} (G)
$$
induce a lattice isomorphism between the filter lattice of $\alg A$ and the filter lattice of $\escl \alg A$.  Since they both clearly preserve the ordering we need only check that they are well defined and their
composition is the identity on the respective domains.

That for any filter $H$ of $\alg A$,  $H \cap \escl A$ is a filter of $\escl \alg A$, it is a consequence of
Lemma \ref{girlemma}(4). Now let $G$ be a filter of $\escl \alg A$ and let
$H = \op{Fil}_{\alg A}(G)$. We will show that $ H \cap \escl A = G$.
Clearly $G \sse H \cap \escl A$.
If $\escl a \in H$, then  by Lemma \ref{filgen} there are $\vuc {\escl b}n \in G$
such that
$$
\escl b_1\dots \escl b_n \le \escl a.
$$
But since
$$
\escl b_1\dots \escl b_n= \escl(b_1 \meet\dots\meet b_n) = \escl b_1 \meet_! \dots \meet_! \escl b_n
$$
by the usual description of filters in a Heyting algebras
we get that $\escl_F a \in G$.

On the other hand if $H$ is a filter of $\alg A$, it is obvious that $\op{Fil}_{\alg A}(H \cap \escl A) =H$ is obvious,
since $\escl a \le a$.
\end{proof}

Let us define an operator $\op{Heyt}$ on a class $\vv K$ of girales by
$$
\op{Heyt}(\vv K)= \{ \escl \alg A: \alg A \in \vv K\}.
$$

\begin{theorem} For any variety $\vv V$ of girales, $\op{Heyt}(\vv V)$ is a variety of (pointed) Heyting algebras.
\end{theorem}
\begin{proof} We need only to show that $\op{Heyt}(\vv V)$ is closed under $\HH,\SU$ and $\PP$. Let $\escl \alg A \in \op{Heyt}(\vv V)$ and
let $G \in \op{Fil}(\escl \alg A)$. Let $F= \{ b \in A: \escl a \le b \ \text{for some}\ \escl a \in G\}$; then $F$ is  a filter of $\alg A$
and $F \cap\escl A = G$.  So $\escl (\alg A/F) \cong \escl \alg A/G$ and $\escl \alg A/G \in \op{Heyt}(\vv V)$. That $\op{Heyt}(\vv V)$ is closed under direct products
is obvious so let $\escl \alg A \in \op{Heyt}(\vv V)$ and let $\alg C$ be a subalgebra of $\escl \alg A$ and let $\alg B = \op{Sub}_\alg A(C)$. Clearly $\alg C \sse \escl \alg B$;
conversely, as $\alg C$ generates $\alg B$ in $\alg A$, every element of $\escl \alg B$ is $\escl t(\vuc cn)$ for some $n$-ary term of $\alg A$ and $\vuc cn\in C$.
Now an induction on the complexity of $t(\vuc xn)$ shows that if $\vuc cn \in C$ then $\escl t(\vuc cn) \in \escl C$ .  The only nontrivial part is to show that
if $c \in C$, then $\escl \nneg \escl c \in C$.

  Observe that  from
$\nneg \escl c \escl c \le \nneg 1$ and $\escl \nneg \escl c \le \nneg
\escl c$ we get $\escl \nneg \escl c  \escl c \le \nneg 1$ and
hence $\escl \nneg \escl c \escl c \le \escl \nneg 1$.

By residuation $\escl \nneg \escl c \le \escl c \imp \escl \nneg 1$
and hence $\escl\nneg \escl c \le \escl(\escl c \imp \escl \nneg 1)=
\escl c \imp_{!} \escl \nneg 1$. On the other hand
$$
\nneg \escl c = \escl c \imp \nneg 1 \ge \escl c \imp \escl \nneg 1,
$$
implying
$$
\escl\nneg \escl c \ge \escl(\escl c \imp \escl \nneg 1)= \escl c \imp_{!}\escl \nneg 1.
$$
In conclusion $\escl\nneg\escl c =\escl c \imp_{!} \escl\nneg 1 \in C$.
\end{proof}

The mapping $\vv V \longmapsto \op{Heyt}(\vv V)$
is a join homomorphism from the lattice of subvarieties of  girales to the lattice of subvarieties of Heyting
algebras:
\begin{align*}
\op{Heyt}(\vv V \join \vv V') &= \op{Heyt}(\VV(\vv V \cup \vv V'))
= \VV(\op{Heyt}(\vv V \cup \vv V'))\\
&=\VV(\op{Heyt}(\vv V) \cup \op{Heyt}(\vv V')) = \op{Heyt}(\vv V) \join \op{Heyt}(\vv
V').
\end{align*}
However it is not a meet homomorphism, as we shall see later.

Thus the lattice of varieties of girales can be
partitioned into equivalence classes that are also join semilattices;
information on these classes can be recovered from the varieties of Heyting
algebras that are their ``natural'' representatives. Such pieces of
information can be
glued together to get a clearer picture of the whole lattice of varieties
of girales.

Let's call a girale $\alg A$ a {\bf Boolean girale} if $\op{Heyt}(\alg A)$ is a Boolean algebra; a variety $\vv V$ of girales
is {\bf Boolean} if every algebra in $\vv V$ is Boolean.
We will give a recipe to construct a family
of simple Boolean
girales of unbounded cardinality.
By J\'onnson Lemma \cite{Jonsson1967}, any two of them
generate  distinct  Boolean varieties of girales.

Let $\alg G_n$ be the height three lattice
with $n$ atoms. Let $\bot$ and $\top$ be the bottom and the top of the
lattice and let $1$, $0$ be two distinct atoms. Define
$\nneg$ on $G_n$ by setting $\nneg \bot = \top$, $\nneg\top =  \bot$,
$\nneg 1 = \bot$, $\nneg 0 = 1$ and $\nneg a =a$ for any other
$a \in G_n$. Define  $\cdot$ on $G_n$
by
\begin{align*}
&\bot \cdot a = a \cdot \bot =\bot  \\
&1 \cdot a = a \cdot 1 = a \\
&a \cdot a = \bot \qquad  a \notin\{0,1,\bot,\top\}\\
& a \cdot b = \top \qquad\text{otherwise}
\end{align*}
Finally define $\escl 1 = \escl \top = 1 $ and $\escl a = \bot$
otherwise. Some calculations show that each $\alg G_n$ is a Boolean girale;
the subdirectly irreducible algebras in $\VV(\alg G_n)$ are exactly
the $\alg G_m$, $1\le m\le n$.

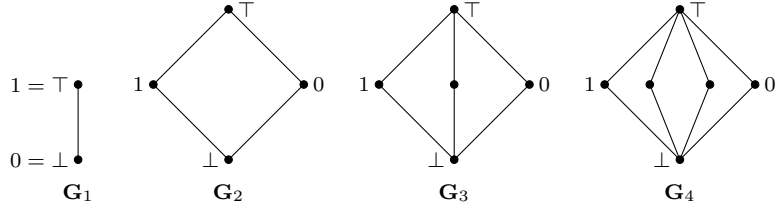
\begin{figure}[htbp]
\begin{center}
\begin{tikzpicture}[scale=1]
\draw (0,0) -- (0,1);
\draw (2,0) -- (1,1) -- (2,2) -- (3,1) -- (2,0) ;
\draw (5,0) -- (4,1) -- (5,2) -- (5,0) -- (6,1) -- (5,2);
\draw  (8,0) -- (7,1) -- (8,2) -- (7.6,1) -- (8,0) -- (8.4,1) -- (8,2) -- (9,1) -- (8,0) ;
\draw[fill] (0,0) circle [radius=0.05];
\draw[fill] (0,1) circle [radius=0.05];
\node[left] at (0,0) {\footnotesize$0=\bot$};
\node[left] at (0,1) {\footnotesize $1=\top$};
\draw[fill] (2,0) circle [radius=0.05];
\draw[fill] (1,1) circle [radius=0.05];
\draw[fill] (2,2) circle [radius=0.05];
\draw[fill] (3,1) circle [radius=0.05];
\node[left] at (2,0) {\footnotesize$\bot$};
\node[left] at (1,1) {\footnotesize $1$};
\node[right] at (2,2) {\footnotesize $\top$};
\node[right] at (3,1) {\footnotesize $0$};
\draw[fill] (5,0) circle [radius=0.05];
\draw[fill] (4,1) circle [radius=0.05];
\draw[fill] (5,2) circle [radius=0.05];
\draw[fill] (6,1) circle [radius=0.05];
\draw[fill] (5,1) circle [radius=0.05];
\node[left] at (5,0) {\footnotesize$\bot$};
\node[left] at (4,1) {\footnotesize $1$};
\node[right] at (5,2) {\footnotesize $\top$};
\node[right] at (6,1) {\footnotesize $0$};
\draw[fill] (8,0) circle [radius=0.05];
\draw[fill] (7,1) circle [radius=0.05];
\draw[fill] (8,2) circle [radius=0.05];
\draw[fill] (7.6,1) circle [radius=0.05];
\draw[fill] (8.4,1) circle [radius=0.05];
\draw[fill] (9,1) circle [radius=0.05];
\node[left] at (8,0) {\footnotesize$\bot$};
\node[left] at (7,1) {\footnotesize $1$};
\node[right] at (8,2) {\footnotesize $\top$};
\node[right] at (9,1) {\footnotesize $0$};
\node[below] at (0,-.2) {\footnotesize  $\alg G_1$};
\node[below] at (2,-.2) {\footnotesize  $\alg G_2$};
\node[below] at (5,-.2) {\footnotesize  $\alg G_3$};
\node[below] at (8,-.2) {\footnotesize  $\alg G_4$};
\end{tikzpicture}
\caption{Boolean girales}\label{girales}
\end{center}
\end{figure}

The algebra $\alg G_1$ in Figure \ref{girales} is the two element Boolean
algebra and hence $\op{Heyt}(\VV(\alg G_1)) = \vv B$, the variety of Boolean algebras. On the other hand,
since both $\alg G_1$ and $\alg G_2$  are finite simple algebras with no proper
subalgebras  both $\VV(\alg B)$ and $\VV(\alg G_2)$ are atoms in the lattice of
varieties of girales and  their intersection is the trivial variety.
However $\op{Heyt}(\VV(\alg G_2)) = \vv B$ as well,  so
$\op{Heyt}(\VV(\alg G_1)) \cap \op{Heyt}(\VV(\alg G_2)) = \vv B$.
This shows that the map $\vv V \longmapsto \op{Heyt}(\vv V)$ is not a meet homomorphism.

\section{Fragments and expansions of $\mathsf{LL}$}

\subsection{Fragments}
Any fragment of $\mathsf{LL}$ whose language contains the connectives
$\imp,\meet$ is algebraizable. This is obvious if $1$ is contained
in the language as well, since in this case, by Corollary 2.12 of
\cite{BlokPigozzi1989}, the defining equation and congruence formulas are the
same as for $\mathsf {LL}$. It is easily checked that these fragments are in
fact strongly algebraizable.

If $1$ is not contained in the language, then we get something very
similar to what happens in  Relevance Logic.
We will describe one case in detail and leave all the others to the
interested reader. Let $\mathsf {LR}$ be the system axiomatized by axioms
(HL1)--(HL7), (HL12)--(HL19), (MP) and (Adj). The system is algebraizable
with defining equation $p \meet (p \imp p)= p \imp p$ and equivalence
formulas $\D= \{p \imp q,q\imp p\}$. The Lindenbaum algebra of
$\mathsf {LR}$ is easily seen to be a bounded lattice with involution and
with an implication satisfying (2.2)--(2.5). Unfortunately the variety $\vv
V$
of such algebras is not the equivalent algebraic semantics of $\mathsf {LR}$,
since the natural
ordering of the underlying lattice does not model adequately the residuation
in $\mathsf {LR}$. It turns out that the variety  of
$\mathsf {LR}$-algebras, i.e.  the subvariety of $\vv V$ axiomatized by the equation
$$
((x \imp x)\meet (y\imp y))\imp z \le z,
$$
is the  equivalent algebraic semantics for $\mathsf{LR}$.

In this context the variety $\vv R$ of relevant algebras is a subvariety
of the variety $\vv H$ of $\mathsf {LR}$-algebras. Since $\vv R$ does not have
equationally definable principal congruences, the variety of $\mathsf{LR}$-algebras  does not have it as
well.
However, if we add to $\mathsf {LR}$ the {\em mingle axiom}
$$
p\imp (p\imp p),
$$
then the resulting equivalent algebraic semantics has equationally definable
principal congruences and thus the system has the deduction theorem. This
displays once
more the strong connections between Linear Logic and Relevance Logic.

The  implicational fragment of $\mathsf CLL$ gives rise to a well known
deductive system, BCI-logic, which is not algebraizable
\cite{BlokPigozzi1989}. In \cite{BlokPigozzi1992}, Blok and Pigozzi noted
also that the $\{\imp,\cdot\}$-fragment of $\mathsf {LL}$ is not algebraizable.
We describe briefly its logical matrices: let $\la A,\imp,\cdot,\le\ra$
be a commutative
partially ordered residuated semigroup. A {\em filter} of $\alg A$ is
an order filter of $A$ closed under multiplication. One sees easily
that a reduced matrix
for the $\{\imp,\cdot\}$-fragment of $\mathsf {LL}$ is $\la A,\nabla_A\ra$
where $\alg A$ is a commutative p.o. residuated semigroup and
$\nabla_A$ is the filter generated by the set $\{a\imp a: a \in A\}$.
Hence the reduced matrix semantics is the class of {\em reduced filtered
commutative residuated partially ordered semigroups}. The reason why this
class cannot be replaced by a class of proper algebras lies in
the fact that the partial order cannot be recovered from the operations.
By the same argument one sees that the multiplicative fragment of $\mathsf {LL}$
is not algebraizable.

Finally any algebraizable fragment whose language contains $\escl$ has the
deduction theorem and  gives rise to an equivalent algebraic semantics
having equationally definable principal congruences.

\subsection{Intuitionistic Linear Logic}
There are at least two  versions of intuitionistic linear logic.
The first one is $\mathsf {ILL}$ (see \cite{Girard1987}) which is obtained from
classical linear logic in the
same way Gentzen's system $\mathsf {LJ}$ (intuitionistic logic) is obtained from
$\mathsf {LK}$: a left side of a sequent can contain at most one formula.
So we are forced to drop all connectives and constants whose rules do not
obey to this restriction, i.e. $\?, \neg, \bot, \Par$. $\mathsf {ILL}$ can be seen
as a deductive system axiomatized Hilbert-style by (HL1)--(HL3),
(HL6)--(HL9),
(HL12)--(HL24) with (MP), (Adj) and (Nec).

It is sometimes convenient, to get a more suitable comparison with classical
linear logic, to
consider $\mathsf {ILN}$ \cite{Abrusci1990}, i.e. {\em intuitionistic Linear
Logic with negation}. This is done by adding two ``ad hoc'' rules for $\nneg$
in the sequent calculus. On the Hilbert-style side we just add (HL5),
(HL10), (HL11) (one of course has to define $\bot$ as $\nneg \one$). Finally,
if we drop (HL20)--(HL24) and (Nec) from $\mathsf {ILL}$ or $\mathsf {ILN}$ we get the
exponential-free fragments.

It is clear that all these systems are fragments of $\mathsf {LL}$ with respect
to
suitable languages. Since the connectives appearing
in the congruence formulas and defining equation for $\mathsf {LL}$ belong to
all these languages, we can conclude at once (by Corollary 2.12 in
\cite{BlokPigozzi1989}) that all the systems are algebraizable (and they in fact  strongly
algebraizable).

\subsection{Noncommutative Linear Logics}
If in the sequent formulation of Linear Logic we delete the
{\em exchange rule}, we get {\em noncommutative Linear Logic}.
It is not hard (only boring) to work out a Hilbert-style axiomatization
of non commutative Linear Logic. In this framework one has to deal with two
implications and two negations and the product $\cdot$ is no
longer commutative. Of course also in this case we have interesting fragments
and in particular one can get {\em noncommutative intuitionistic linear
logic} \cite{Abrusci1991}. Noncommutative Linear Logics are algebraizable
(the proof is similar to the one for Linear Logic, only more work
has to be done to take care of the doubling of implication and
negation)
and the equivalent algebraic semantics are varieties of $\mathsf{FL}$-algebras
with further operations.

\providecommand{\bysame}{\leavevmode\hbox to3em{\hrulefill}\thinspace}
\providecommand{\MR}{\relax\ifhmode\unskip\space\fi MR }
\providecommand{\MRhref}[2]{%
  \href{http://www.ams.org/mathscinet-getitem?mr=#1}{#2}
}
\providecommand{\href}[2]{#2}

\end{document}